\documentclass[leqno]{article}

\usepackage{amsmath}
\usepackage{amsthm}
\usepackage{amssymb}
\usepackage{amsfonts}
\usepackage{hyperref}
\usepackage{multirow}
\usepackage{booktabs}

\usepackage{color}
\usepackage{listings}

\usepackage{marginnote}

\usepackage{pgfplots}

\author{Francisco Jos\'e Vial Prado \thanks{DCC,  Pontificia
Universidad Cat\'olica de Chile (\texttt{fovial@uc.cl}).}}
\date{}
\numberwithin{equation}{section}

\newcommand{\RR}{\mathbb{R}}      

\newcommand{\ep}{\varepsilon}
\newcommand{\dist}{\mbox{dist}}
\newcommand{\EE}{\mathcal{E}}      
\newcommand{\AAA}{\mathcal{A}}      

\newtheorem{definition}{Definition}
\newtheorem{proposition}{Proposition}
\newtheorem{theorem}{Theorem}
\newtheorem{lemma}{Lemma}

\newcommand{\ie}{i.e.$\!$ }

\begin{document}

\title{The Gelfand Problem in Tubular Domains}

\maketitle

\abstract{
    We construct stable solutions of $\Delta u + \lambda e^u=0$ with
    Dirichlet boundary conditions in small tubular domains (i.e. geodesic
    $\ep$--neighbourhoods of a curve $\Lambda$ embedded in $\RR^n$), adapting
    the arguments of Pacard-Pacella-Sciunzi. We also show unicity of these
    solutions, in particular, we show that the stable branch of the bifurcation
    diagram is similar to the well-known nose-shaped diagram of the standard
    Gelfand problem in the unit ball. In this work, $\Lambda$ can be replaced
    by any compact smooth manifold embedded in $\RR^n$.
}


\smallskip



\section{Introduction}
Let $\Omega$ be a domain of $\RR^n$, and for $\lambda>0$ consider the Gelfand 
problem
\begin{equation}
    \label{eq-g-lambda}
(G_\lambda)\;\left\{ 
    \begin{array}{cl} \Delta u + \lambda e^u=0 & \mbox{ in }\Omega,\\
        u=0 & \mbox{ on }\partial \Omega.
    \end{array}\right.
\end{equation}
Also, let $\Lambda$ be a smooth closed curve embedded in $\RR^n$, and given
$\ep>0$ we define the tubular neighbourhood of radius $\ep$ centered about
$\Lambda$ by
\begin{equation}
    \label{eq-tube}
    T_\ep(\Lambda):= \{x\in \RR^n: \dist(x,\Lambda)<\ep\}.
\end{equation}

In this article, we study the Gelfand problem $(G_\lambda)$ in this kind of
tubular domains. Let us first give some properties of this problem
in all generality, assuming that $\Omega$ is
bounded and connected.

\begin{itemize}
    \item[(i)] If $u$ solves $(G_\lambda)$, $\lambda>0$ if and only if $u>0$.
        This is a straightforward consequence of the maximum principle.
    \item[(ii)] If $(G_\lambda)$ admits a solution $u$, then
        $\lambda \leq \lambda_1(\Omega)$, where $\lambda_1(\Omega)$
        stands for the first eigenvalue of the Dirichlet Laplacian in $\Omega$.
        Indeed, let $\phi>0$ be a principal eigenfunction of the Laplacian in
        $\Omega$ with Dirichlet boundary conditions, and multiply $(G_\lambda)$ by
        $\phi$. Integrating by parts and using $u<e^u$ gives
        $(-\lambda_1(\Omega)+\lambda)\int_\Omega \phi\,u < 0.$
    \item[(iii)] With implicit function methods, one can establish a local
        solution curve $(\lambda, u)\in [0,\infty)\times C(\bar\Omega),$ which
        emanates from the stable solution $\lambda=0,u=0$. Because of (ii) this
        curve is contained in $[0,\lambda_1(\Omega)]\times C(\bar\Omega)$.
\end{itemize}

We can therefore identify $(\lambda,u)\in [0,\lambda_1]\times C^2(\bar\Omega)$
with a pair $(\lambda, ||u||_{L^\infty})\in[0,\lambda_1]\times \RR^+$, and draw
the bifurcation diagram in the plane. However, determining multiplicity of these
solutions require an understanding in the general domain $\Omega$. Nevertheless,
theorem 4.6 of \cite{isoperimetric-inequalities} implies that the spectrum of
$(G_\lambda)$ (\ie the values for which $(G_\lambda)$ is solvable) is an
interval $[0,\lambda_\ast(\Omega))$, the value $\lambda_\ast$ depending heavily
on the geometry of $\Omega$.

\section{Preliminaries}
\label{sec-preliminaries}


\subsection{The stable branch}
\label{sec:stable-branch}

Let $L(u):=-\Delta u - \lambda e ^ u$. We say that a solution $u$ of
$(G_\lambda)$ is \textit{stable} if the linearized operator of $L$ at the point
$u$ is positive definite. This notion of stability comes from the following
argument: For $\Omega$ consider the functional $\EE:C_0^2(\Omega)\to\RR$
defined by
\begin{equation}
    \EE(u):= \frac{1}{2}\int_\Omega |\nabla u|^2\, dx - \int_\Omega e^u\, dx.
\end{equation}
We say that $u$ is a critical point of $\EE$ if for every $\phi\in
C_0^2(\Omega)$, $0$ is a critical point of $E:\RR\to \RR$ given by
$E(t):=\EE(u+t\phi)$. Indeed, the equation $E'(0)=0$ gives
\begin{equation}
    \int_\Omega (-\Delta u-e^u)\,\phi \, dx = 0\; \mbox{ and }\;
    E''(0)=\int_\Omega|\nabla u|^2 - \int_\Omega e^u\phi^2.
\end{equation}

It is hence natural to define stability as follows.

\begin{definition} Let $\Omega$ be an open set of $\RR^n$, and $u\in
    C_0^2(\Omega)$ a solution of $-\Delta u = f(u)$. We say that $u$ is stable
    if
    $$
    Q_u(\phi):=\int_\Omega|\nabla \phi|^2\;dx - \int_\Omega f'(u)\,\phi^2\;dx
    \geq 0
    $$
    for all $\phi\in C_c^1(\Omega)$ (or $\phi\in H_0^1(\Omega)$ if $\Omega$ is
    bounded).
\end{definition}

For the sake of completeness, let us give some properties of stable solutions.

\begin{proposition} Local minimisers of the energy $\EE$ are stable.
\end{proposition}

\begin{proposition} A $C_0^2(\Omega)$ solution of $-\Delta u = f(u)$ is stable
    if and only if $\lambda_1(-\Delta u-f'(u),\omega)\geq 0$ for every bounded
    subdomain $w$ of $\Omega$ (or simply, $\omega=\Omega$ if $\Omega$ is
    bounded).
\end{proposition}
\begin{proposition} A $C_0^2(\Omega)$ solution of $-\Delta u = f(u)$ is stable
    if and only if it exists $v\in C^2(\Omega)$, $v>0$, and $-\Delta v -f'(u)\,
    v\geq 0.$
\end{proposition}
\begin{proposition} There is one unique stable solution of $(G_\lambda)$ for
    every admissible $\lambda$, which is a minimiser of the energy and the
    solution with the smallest $L^\infty(\Omega)$ norm.
\end{proposition}

For more properties and proofs we recomend \cite{stable-solutions-elliptic},
$\S 1$. We define also the lower (or stable) branch of the curve
$$\mathcal S:=\{(\lambda,||u_\lambda||_{L^\infty(\Omega)}),\, u_\lambda\mbox{ is a
    solution of }(G_\lambda)\}.
$$
It turns out (see \cite{stable-solutions-elliptic}, \S 3.3), that for every $n$
there exists a maximum value $\lambda_\ast$ such that $\mathcal{S}$
is a smooth curve connecting $(0,0)$ with $(\lambda_\ast,
||u_{\lambda_\ast}||_{L^\infty(\Omega)})$, which is unbounded in the $||\cdot
||_{L^\infty}$ direction if $n\geq 10$.


\subsection{The case $\Omega=B^n$}
\label{sec:ball}
The celebrated theorem of Gidas-Ni-Nirenberg (see \cite{gidas-ni-nirenberg})
establishes that the solutions of
\begin{equation}
    (\mbox{GNN})\;\left\{ 
        \begin{array}{cc} \Delta u + f(u)=0 & \mbox{ in }B^n,\\
        u=0 & \mbox{ on }\partial B^n,
    \end{array}\right.
\end{equation}
are positive and radially symmetric provided that $f$ is positive and regular,
where $B^n=\{x\in\RR^n,|x|\leq 1\}.$ Therefore, if $\Omega=B^n$, all solutions
of $(G_\lambda)$ are radially symmetric and $(G_\lambda)$ is equivalent to the
problem of finding $u:[0,1]\to\RR$ such that
\begin{equation}
    \left\{\begin{array}{lc} 
            u''+\frac{n-1}{r}u'+\lambda\, e^u =0, & r\in (0,1),\\
    u'(0) = u(1) = 0.&  \end{array}\right.
\end{equation}
Notice that $||u||_{L^\infty}=u(0)$. For $n=1$, the equation $u''+\lambda e^u=0$
may be explicitely solved by means of the Laplace transform (see \cite{bratu}),
giving $u(x)=2\,\log\big(\alpha\,\mbox{sech}(\alpha\sqrt{\lambda/2}\cdot
x)\big)$, and the boundary conditions require that $\alpha$ is a solution to
the transcendental equation $\alpha
=\mbox{cosh}\big(\alpha\sqrt{\lambda/2}\big)$. One verifies that this leads to
zero, one, or two solutions for $\lambda>\lambda_c,\lambda=\lambda_c$ and
$\lambda<\lambda_c$ respectively, where $\lambda_c\approx .88$. Observe that
$||u||_{L^\infty}=u(0)=2\,\log\alpha$, and from the transcendental equation one
computes $\alpha_1=1+\frac{\lambda}{4}+o(1)$, $\alpha_2=\frac{4}{\lambda}+o(1)$.
It thus follows that, for any small and fixed $\lambda$, one solution
approaches to $||u||=0$ and the other one approaches to $||u||=\infty$, this is,
the solution curve $(\lambda, ||u_\lambda||)$ is unbounded, contained in
$[0,\lambda_c]\times \RR^+$. One can easily check that, in the upper branch of
solutions, $\lambda$ decreases with $||u||$, yielding the nose-shape of the
solution curve.

Other branches of solutions have been computed by numerical means, see for
instance \cite{new-solutions}.  The bi-dimensional case $\Omega=B^2$ can also
be solved explicitly and it presents a similar behaviour than the previous one.
Solutions exist if and only if $0\leq \lambda \leq 2$, and for $\lambda=2$
there is only one solution given by $u_\ast(r)=\log\frac{4}{(1+r^2)^2}$.
For other admissible values of $\lambda$, solutions are given by
$$
u_i(r)=\log\frac{b_i}{(1+(\lambda\, b_i/8)r^2)^2},
$$
where
$b_i=\frac{32}{\lambda^2}\big(1-\frac{\lambda}{4}+(-1)^i\sqrt{1-\lambda/2}\big),
i=1,2$.
\medskip

There is, however, a remarkable difference between $n=1$ and $n=2$: For
$\Omega=(-1,1)$, the unstable solution blows up at every point as $\lambda\to
0$, whereas for $\Omega=B^2$, it blows up only at the origin.
\medskip

For general dimension $n\geq 3$, the problem is analysed with a suitable change
of variables, and the behaviour of the solution yields from a dynamical coupled
system that arises from the study of the radial equation. This is the approach
considered by Dupaigne in \cite{stable-solutions-elliptic}; for other
instances of the Gelfand problem we strongly recommend this text. We summarize
the results for the radially symmetric case in every dimension in Table \ref{table:bifurcation-ball}.

\begin{table}[ht!]
    \center
    \renewcommand{\arraystretch}{1.2}
    \begin{tabular}{|c|p{.4\textwidth}|c|}
\hline
\textit{Dimension}                 & \textit{Number of solutions}                                            & \textit{Maximum value of $\lambda$}                  \\ \hline
\multirow{2}{*}{$n=1,2$}           & Two if $\lambda\in(0,\lambda_\ast)$                                     & \multirow{8}{*}{$\lambda_\ast(B^n)>2(n-2)$} \\ \cline{2-2}
                                   & One if $\lambda=\lambda_\ast$                                           &                                             \\ \cline{1-2}
\multirow{6}{*}{$3\leq n  \leq 9$} & One if $\lambda$ is sufficiently small                                  &                                             \\ \cline{2-2}
                                   & Finitely many if $\lambda\neq 2(n-2)$                                   &                                             \\ \cline{2-2}
                                   & More than any given number for $\lambda$ sufficiently close to $2(n-2)$ &                                             \\ \cline{2-2}
                                   & Infinitely many for $\lambda = 2(n-2)$                                  &                                             \\ \cline{2-2}
                                   & Two for $\lambda$ close to $\lambda_\ast(B^n)$                          &                                             \\ \cline{2-2}
                                   & One for $\lambda = \lambda_\ast(B^n)$                                   &                                             \\ \hline
$n\geq 10$                         & One unique stable solution                                              & $\lambda_\ast(B^n)=2(n-2)$                  \\ \hline
\end{tabular}
\caption{Bifurcation diagram for the Gelfand problem in $\Omega = B^n$ (\cite{stable-solutions-elliptic})}
\label{table:bifurcation-ball}
\end{table}


\subsection{Fermi coordinates}
\label{sec:fermi}

Our framework is a specific case of the one used in \cite{Pacard2014} for a
1-dimensional manifold, and we will indeed follow their notations. Consider the
tubular neighbourhood $T_\ep$ as defined in \ref{eq-tube}; The Fermi coordinates
parameterise this set as a product space between the curve $\Lambda$ and
$B^{n-1}$ as follows. First identify $\Lambda$ with the zero-section of
$N\Lambda$ (the normal bundle of $\Lambda$) and $T_\ep$ with
$$
\Omega_\ep:=\{(y,z)\in N\Lambda;\, y\in\Omega,\, z\in N_y\Lambda,\, |z|\leq\ep\}
$$
via the natural mapping $T_\ep\to\Omega_\ep$, $(y,z)\mapsto y+z$.

If $g_z:=dz^2$ is the Euclidean metric on normal fibers, and $\mathring{g}$ is
the metric induced on $\Lambda$, we have that the metric $\bar g$ on $N\Lambda$
is induced by the embedding of $\Lambda$ in $\RR^n$, this is $\bar g=\mathring
g+g_z$. Lemma 3.2 of \cite{Pacard2014} proves that, in these coordinates, the
Euclidean Laplacian $\Delta$ can be decomposed as
\begin{equation}
\Delta = \Delta_{\bar g} + D,
\end{equation}
where $\Delta_{\bar g}=\Delta_{\mathring g}+\Delta_{g_z}$ denotes the
Laplace-Beltrami operator on $N\Lambda$ for metric $\bar g$, and $D$ is a
second-order differential operator of the form
\begin{equation}
D = \sum_{i=1}^{n-1} z_iD_i^{(2)}+D^{(1)},
\end{equation}
where $D^{(1)}$ and $D^{(2)}$ are first-order and second-order partial
differential operators respectively, whose
coefficients are smooth and bounded.
Note that in this $1$--dimensional manifold, we have a parameterisation
$t:\RR\to\Lambda$ of the curve, and the Laplace-Beltrami operator of $\Lambda$
is simply $\partial_{tt}$. We therefore have
\begin{equation}
    \Delta = \partial_{tt}+\Delta_{g_z}+D.
\end{equation}
This decomposition, and in particular the form of the operator $D$, will allow
us to obtain estimates for functions defined in small tubes and prove our main
results, which we describe in the next sections.

\section{The stable solution in the tube}
\label{sec:stable-sol-in-tube}

In this section, we establish the following:

\begin{theorem}
\label{thm:unique-stable-sol}
For small $\ep$, there exists a unique stable solution to
\begin{equation}
\label{eq:small-gelfand}
\left\{
\begin{array}{cc}
\Delta u + \frac{\lambda}{\ep^2}e^u=0 & \mbox{ in }T_\ep(\Lambda),\\
u = 0 & \mbox{ on }\partial T_\ep(\Lambda).
\end{array}
\right.
\end{equation}
\end{theorem}

\noindent An outline of the proof follows:

\begin{itemize}
\item Let $U$ be the unique radial stable solution of $(G_\lambda)$ for the
$(n-1)$-dimensional ball. We choose $u_\ep$, a rescaled version of $U$ that
travels along $\Lambda$, this is, for each section of the normal
bundle of $\Lambda$, consider the $(n-1)$-dimensional ball of radius $\ep$
centered about $\Lambda$ and use a copy of $U$.
\item The function $u_\ep$ verifies approximately the equation and it is stable
in the sense of section \ref{sec:stable-branch}. To prove this, we find a
super-solution to the problem and use a version of the maximum principle.
\item We use a fixed-point argument to prove the existence of a genuine solution
to \ref{eq:small-gelfand} of the form $u_\ep+v$, where $v$ is small. The
fixed-point theorem ensures the uniqueness of the perturbation if $v$ is small
enough. Finally, we show that larger perturbations do not lead to stable
solutions, proving the result.
\end{itemize}

Let us first collect some lemmas that will help us prove
\ref{thm:unique-stable-sol}.
\subsection{Some elliptic estimates}
Throughout this section, $||\cdot||$
stands for the $L^\infty$ norm in the tube
$T_\ep$, unless stated otherwise. Let $U(r)$ be the unique (radial) stable
solution of $(G_\lambda)$ for $\Omega
= B^{n-1}$, \ie
\begin{equation}
\left\{\begin{array}{cc}
\Delta U + \lambda\, e^U=0 & \mbox{ in }B^{n-1},\\
U = 0 & \mbox{ on }\partial B^{n-1},
\end{array}\right.
\end{equation}

\noindent and define $u_\ep:T_\ep\to \RR$ by
$$
u_\ep(y,z):=U\bigg(\frac{\mbox{dist}(z,\Lambda)}{\ep}\bigg).
$$
As the non-linearity $\lambda\,e^u$ is not multiplicative, we have that the
function $u_\ep$ will not verify the Gelfand equation in the tube, nor an
approximation, but the following estimate.

\begin{lemma}\label{lem-bound_Cepsilon} There is a constant $C$ such that
\begin{equation}
|\ep^2\Delta u_\ep+\lambda\, e^{u_\ep}|\leq C\ep.
\end{equation}
\end{lemma}

\begin{proof} It follows from the definition of $u_\ep$ that
$\ep^2\Delta_{g_z}u_\ep + \lambda\, e^{u_\ep}=0$, and note that $u_\ep$ does not
depend on the parameter $t$. Using the expression of the Euclidean Laplacian in
Fermi coordinates, we have
$$\ep^2\, \Delta u_\ep+\lambda e^{u_\ep} = (\ep^2\, \Delta_{g_z} u_\ep +
\lambda e^{u_\ep})+\ep^2(\partial_{tt}+D)u_\ep = \ep^2 D\, u_\ep,
$$
therefore, the estimate
\begin{eqnarray}
|\ep^2\Delta u_\ep + \lambda e^{u_\ep}|&\leq& \ep^2\, ||Du_\ep|| \\
&\leq& \sum_{i=1}^{n-1} \big(\ep^2||z_iD^{(2)}u_\ep||+\ep^2||D^{(1)}u_\ep||\big)\\
    &=& \sum_{i=1}^{n-1}\big(\ep
    ||x_iD^{(2)}U||_{L^\infty(B)}+\ep||D^{(1)}U||_{L^\infty(B)}\big)
\end{eqnarray}
holds and the result follows.
\end{proof}
\medskip

The stability of $U$ implies that there exist $\mu_1>0$ and $\phi_1>0$ such that
\begin{equation}
\left\{\begin{array}{cc}
-(\Delta+\lambda e^U)\phi_1=\mu_1\phi_1 & \mbox{ in }B^{n-1},\\
\phi=0 & \mbox{ on }\partial B^{n-1},
\end{array}\right.
\end{equation}
and hence the linearized operator about $U$ is invertible. This proves the
existence of a function $W$, which is a super-solution to the linearized
equation in $B^{n-1}$, verifying
\begin{equation}
\left\{\begin{array}{cc}
-(\Delta+\lambda e^U)W=1 & \mbox{ in }B^{n-1},\\
W=0 & \mbox{ on }\partial B^{n-1}.
\end{array}\right.
\end{equation}

\begin{lemma} There is a constant $B$ such that $|W|\leq B, |\nabla W|\leq
B, |\Delta W|\leq B$.
\end{lemma}
\begin{proof}
By the GNN theorem, $W$ is radially symmetric, positive and $\partial_r W<0$,
which gives $|W|\leq W(0)$. Write $W(z)=a(r)$, then $a$ solves
\begin{equation}
\left\{
\begin{array}{lc}
a''(r)+\frac{n-1}{r^2}a'(r)+\lambda e^Ua+1=0, & 0<r<1,\\
a(1) = 0, a'(0) = 0. &
\end{array}
\right.
\end{equation}
As $a(r)$ is bounded and positive, we have 
\begin{equation}
    \label{eq:bounding-a'}
-C\leq a''(r)+\frac{n-1}{r^2}a'(r)\leq -1
\end{equation}
for a positive constant $C$.
Multiplying \ref{eq:bounding-a'} by $e^{-\frac{n-1}{r}}$ we have $
-Ce^{-\frac{n-1}{r}}\leq \big(e^{-\frac{n-1}{r}}a'(r)\big)'\leq
e^{-\frac{n-1}{r}}$, \ie  
\begin{eqnarray*}
    |a'(r)|\leq  C_1\cdot e^{\frac{n-1}{r}}\int_0^re^{-\frac{n-1}{t}}\, dt \leq
    C_1  
\end{eqnarray*}
for a constant $C_1=\max(C,1)$. This yields $|\nabla W(z)|=|a(r)|\leq C_1$.
Finally, write $|\Delta W|\leq 1+e^UW\leq C_2$ for a positive constant $C_2$
and define $B:=\max(C_1,C_2,W(0))$.

\end{proof}

Similarly, define $w_\ep:T_\ep\to \RR$ by
$w_\ep(y,z):=W\big(\frac{\operatorname{dist}(z,\Lambda)}{\ep}\big)$. The
corresponding estimate for $w_\ep$ will be given by the following. 

\begin{lemma}
For sufficiently small $\ep$, the function $w_\ep$ verifies
\begin{equation}
    \label{eq-lemma--1/2}
\ep^2 \Delta w_\ep + \lambda e^{u_\ep}w_\ep \leq -1/2.
\end{equation}
\end{lemma}

\begin{proof}
Again note that $w_\ep$ does not depend on the parameter $t$ and that
$-(\ep^2\Delta_{g_z}w_\ep + e^{u_\ep}w_\ep)=1$. We then have that $(\ep^2\Delta
+ \lambda e^{u_\ep})w_\ep = (\ep^2\Delta_{g_z}+\lambda
e^{u_\ep})w_\ep+\ep^2(\partial_{tt} + D)w_\ep = -1+\ep^2Dw_\ep$. As $|\ep^2
Dw_\ep|\leq \ep C'$ for a constant $C'$ (as in the proof of lemma 1)
, the result follows for small $\ep$.
\end{proof}

Consider now the invertibility problem for the linearized operator
$L_\ep:=-(\ep^2 \Delta + e^{u_\ep})$ in $T_\ep(\Lambda)$: Given a smooth $f$ in
the tube, find $\phi$ such that
\begin{equation}
    \label{eq-lin-ep}
    \left\{
    \begin{array}{cc}
        L_\ep\phi = f & \mbox{ in }T_\ep(\Lambda),\\
        \phi = 0& \mbox{ on } \partial T_\ep(\Lambda).
    \end{array}
    \right.
\end{equation}

\begin{lemma} Let $\phi$ solve \ref{eq-lin-ep}, then
    \label{lem-phi}
    $\phi\leq 2 ||f||_{L^\infty(T_\ep)}\cdot w_\ep. $
    Moreover, $L_\ep$ is invertible given that $\ep$ is small enough.
\end{lemma}

\begin{proof} Let us first give the a-priori estimate by means of the maximum
    principle. From equations \ref{eq-lemma--1/2} and \ref{eq-lin-ep} we have
    \begin{equation}
        L_\ep w_\ep \geq \frac{1}{2}\;\mbox{ and }\;
        L_\ep\bigg(\frac{\phi}{2||f||}\bigg)\leq \frac{1}{2}.
    \end{equation}
    By addition, we have $L_\ep\big(w_\ep-\frac{1}{2||f||}\phi\big)\geq 0$.
    Recall the strong maximum principle: If $u$ is a smooth function verifying
    $-\Delta u\geq 0$ in a connected domain $\Omega$ of $\RR^n$, then if $u$
    attains a minimum in $\Omega$, $u$ is constant. The same conclusion holds
    if $-\Delta$ is replaced by $-\Delta+a(x)$ with $a\in L^p(\Omega)$ and
    $p\leq n/2$ (see \cite{1980JPhA...13..417H} for a proof of this theorem with
    applications to a Helium-like system). Note that the conclusion holds
    regardless of the sign of $a(x)$. Apply the strong maximum principle to the
    operator $L_\ep$ to conclude that either $w_\ep-\frac{1}{2||f||}\phi$ is
    constant or it attains a minimum in $\partial T_\ep(\Lambda)$. Because both
    $w_\ep$ and $\phi$ vanish on the boundary, in both cases
    $w_\ep-\frac{1}{2||f||}\phi\geq 0$, proving the first assertion.
    To prove that $L_\ep$ is invertible, write the
    problem in Fermi coordinates $(y,z)\in\Lambda\times B^{n-1}$:
    \begin{equation}
        \left\{
            \begin{array}{cc}
                -(\ep^2 \Delta_{\bar g}+\lambda e^{u_\ep})\phi - \ep^2\, D=f & \mbox{ in }T_\ep(\Lambda),\\
                \phi=0 & \mbox{ on }\partial T_\ep(\Lambda).
            \end{array}
            \right.
    \end{equation}

    Recall that $D$ is of the form $\sum_{i=1}^{n-1} z_iD^{(2)}+D^{(1)}$ where
    $D^{(i)}$ is an $i$--differential operator with smoothly bounded
    coefficients. In order to work in a domain with no dependence in $\ep$, use the
    scaling $z\mapsto z/\ep$ and define $\Phi(y,z):=\phi(y,\ep z),
    F(y,z):=f(y,\ep z)$.
    After simple manipulation, the problem reads
    \begin{equation}
        \left\{
            \begin{array}{cc}
                (L-\ep D)\Phi = F& \mbox{ in }T_1(\Lambda),\\
                \Phi=0 & \mbox{ on }\partial T_1(\Lambda).
            \end{array}
            \right.
    \end{equation}
    where $L=-(\Delta_{\bar g} + \lambda e^U)$ is invertible by hypothesis.
    Indeed, its eigenvalues are numbers of the form $\mu_i+\ep^2 \nu_j$ where
    $0<\mu_1\leq \mu_2\leq \mu_3\leq \dots$ are the eigenvalues of
    $-(\Delta_{g_z}+\lambda e^U)$ in $B^{n-1}$ and $0=\nu_1<\nu_2\leq \nu_3\leq
    \dots$ are the eigenvalues of $-\Delta_{\mathring g}$ on $\Lambda$.
    Therefore, $L-\ep D$ is a small perturbation of $L$, and we conclude that
    it is invertible for small $\ep$ after a simple fixed-point argument:
    Write $\Phi+L^{-1}F+\Psi$, then $\Psi$ solves 
    \begin{equation}
        \left\{
            \begin{array}{cc}
                \Psi = \ep L^{-1}D(L^{-1}F+\Psi)& \mbox{ in }T_1(\Lambda),\\
                \Psi=0 & \mbox{ on }\partial T_1(\Lambda).
            \end{array}
            \right.
    \end{equation}
    Using the fact that the norm of the inverse of $L$ is controlled by
    $1/\mu_1$, it is straightforward to show that the right-hand operator (to
    which we associate the $||\cdot||_\infty$ norm in the tube) is a
    contraction mapping for small $\ep$ that maps the space of bounded
    functions in the tube into itself. Thus $\Psi$ exists, allowing to conclude.
\end{proof}

\subsection{Proof of theorem \ref{thm:unique-stable-sol}}

\textit{Proof:} Write a perturbation of the solution $u=u_\ep+v$. Problem
    \ref{eq:small-gelfand} reduces to finding $v$ such that
    \begin{equation}
        \label{eq-fixed-point-H}
        \left\{
            \begin{array}{cc}
                \Delta (u_\ep+v)+\frac{\lambda}{\ep^2}e^{u_\ep+v}=0& \mbox{ in }T_\ep(\Lambda),\\
                v=0 & \mbox{ on }\partial T_\ep(\Lambda).
            \end{array}
            \right.
    \end{equation}
    Rewrite the differential equation as $(\ep^2\Delta + \lambda
    e^{u_\ep})v+(\ep^2\Delta u_\ep + \lambda e^{u_\ep})+\lambda
    e^{u_\ep}(e^v-1-v)=0$, and, recalling that $L_\ep$ is invertible for
    small $\ep$, define the following operator
    \begin{equation}
        H_\ep:= v\mapsto L_\ep^{-1}((\ep^2\Delta u_\ep + \lambda
        e^{u_\ep})+\lambda e^{u_\ep}(e^v-1-v)).
    \end{equation}
    Thus, any solution of \ref{eq-fixed-point-H}  is a fixed point of $H_\ep$. Let us introduce the space of functions
    $$\AAA_\ep :=\{v\in L^\infty(T_\ep(\Lambda)),\exists\, C \in \RR, ||v||\leq
    C\ep\},$$
    to which we associate the $L^\infty$ norm in the tube.
    The following lemmas show that $H_\ep$ is a contraction mapping in
    $\AAA_\ep$. First, the fact that $H_\ep(\AAA_\ep)\subset \AAA_\ep$ for
    sufficiently small $\ep$ is a direct consequence of Lemma \ref{lem-phi} for
    a function in $\AAA_\ep$. 
    \begin{lemma}
        Let $v$ verify \ref{eq-fixed-point-H} and define $B=\max(1,2||\lambda
        e^U W||_{L^\infty(B^{n-1})})$. If $|v|\leq 1/B$, then $|v|\leq C\ep
        w_\ep$, where $C$ is a constant close to the constant from Lemma
        \ref{lem-bound_Cepsilon}.
    \end{lemma}
    \begin{proof} Lemma \ref{lem-phi}, applied for $v$ and $f=(\ep^2\Delta
        u_\ep+\lambda e^{u_\ep})+ \lambda e^{u_\ep}(e^v-1-v)$ gives
    \begin{eqnarray}
        v &\leq & 2 ||(\ep^2\Delta u_\ep + \lambda e^{u_\ep})+\lambda
        e^{u_\ep}(e^v-1-v)||w_\ep\\
          &\leq & C\ep w_\ep + B(e^v-1-v),
    \end{eqnarray}
    for a constant $C$ from Lemma \ref{lem-bound_Cepsilon} not depending on
    $\ep$. Using the fact that $e^t-1-t<t^2$ for $t\in[-1,1]$ (the first
    non-zero root of $g(t):=t^2-(e^t-1-t)$ is greater than $\ln(3)$), we
    conclude that, as $|v|\leq 1$, $|v|\leq C\ep w_\ep+Bv^2$. Thus,
    \begin{equation}
        \big(|v|-C\ep w_\ep+O(\ep^2)\big)\cdot\big( |v| -\frac{1}{B}-C\ep
        w_\ep+O(\ep^2)\big)\geq 0,
    \end{equation}
    from where the result follows. 
\end{proof}

\begin{lemma}
    $H_\ep$ is a contraction mapping of $\AAA_\ep$.
\end{lemma}
\begin{proof}
    Take $f,g\in \AAA_\ep$, then,
    $$||H_\ep(f-g)||=||L_\ep^{-1}\big(\lambda e^{u_\ep}(e^f-e^g -
    (f-g))\big)||.  $$
    The convexity of the exponential function implies that for any reals
    $\alpha,\beta$ with $\alpha>\beta$ we have $e^{\alpha}-e^{\beta}\leq
    e^{\alpha}(\alpha-\beta)$, and, taking $\beta=0$, $e^\alpha-1\leq \alpha e^\alpha$, therefore,
    \begin{eqnarray*}
        ||H_\ep(f-g)||&\leq& C\lambda ||e^{u_\ep}||\cdot ||e^{\max(f,g)}-1||\cdot
    ||f-g||\\
                      &\leq& C'\ep||f-g||,
    \end{eqnarray*}
    for a constant $C'$ not depending on $\ep$.
\end{proof}
This proves that there exists a
    unique genuine solution of $\ref{eq:small-gelfand}$ of the form $u_\ep+v$
    with $v$ small. The stability of this solution comes from the fact that the
    spectrum of the operators $-(\Delta+\lambda e^U)$ in $T_1(\Omega)$ and
    $-(\Delta + \lambda e^{u_\ep}$ in $T_\ep(\Omega)$ are equal: We have
    $$-(\Delta + \lambda e^{u_\ep+v})=-(\Delta + \lambda e^{u_\ep})+O(\ep)
    \mbox{ in } T_\ep(\Omega),$$
    and therefore the eigenvalues of the left-hand operator are close to the
    sequence $\mu_i$, which does not depend on $\ep$. In particular, they form
    a sequence of positive values for small $\ep$, which implies precisely that
    the associated quadratic form is positive definite. This allows to conclude
    that, for sufficiently small $\ep$, there exists a unique small perturbation
    that leads to a genuine stable solution of \ref{eq:small-gelfand}. This
    uniqueness holds only in a neighbourhood of $u_\ep$; it remains the
    question of whether there are stable solutions far from $u_\ep$. We answer
    negatively with the following argument: Suppose that there are two distinct
    stable solutions of \ref{eq:small-gelfand}. Their difference $v:=u_2-u_1$
    verifies
    $$-\ep^2 \Delta v = \lambda(e^{u_2}-e^{u_1}).$$
    Multiplying the above equation by the positive part of $v$ and integrating
    in the tube gives
    \begin{equation}
        \label{eq:unique-stable}
        \ep^2\int_{T_\ep}|\nabla v_+|^2\,dx = \lambda
    \int_{T_\ep}(e^{u_2}-e^{u_1})\cdot v_+\, dx.
\end{equation}
    As $u_2$ is stable, we have $\ep^2\int_{T_\ep}|\nabla v_+|^2\geq \lambda
    \int_{T_\ep} e^{u_2}\cdot v_+^2\, dx$. Plugging this inequality into
    \ref{eq:unique-stable} yields
    $$
    0 \leq \lambda \int_{T_\ep}(e^{u_2}-e^{u_1}-e^{u_2}v_+)v_+\, dx.
    $$
    Note that the integrand is a negative number by strict convexity of the
    exponential function, therefore, $v_+=0$. Changing $u_1$ and $u_2$ gives
    $v_-=0$, completing the proof of Theorem \ref{thm:unique-stable-sol}.
    \hfill $\blacksquare$

\section{Concluding remarks}

In this article, we adapted the arguments of
\cite{Pacard2014,stable-solutions-elliptic} to the Gelfand problem $\Delta u +
\lambda e^u=0$ with Dirichlet boundary conditions in a small geodesic tube
around a smooth manifold embedded in $\RR^n$.
We believe that the same line of thinking in \cite{Pacard2014} can be
carried out without further complications here to complete the bifurcation diagram.
This is, to show that the unstable
solutions of the Gelfand problem in the unit ball can also be used to construct
unstable solutions in tubular domains, for $\ep$ possibly outside a set of
resonant values in which the linearized operator $-\Delta - \lambda e^{u_\ep}$ 
is not invertible, where $u_\ep$ is the rescaling of an unstable solution. Moreover,
an analogous spectral analysis of the linearized operator should yield the
Morse index of unstable solutions. This way, we expect the same resonance phenomena,
\ie, that the invertibility of linearized operators holds only for values of $\ep$
in a set accumulating around 0 (but not for every small $\ep>0$).

\medskip 

\textit{Acknowledgements: During this work, the author was guest
    researcher at the Facultad de Matem\'aticas de la Pontificia Universidad
    Cat\'olica de Chile. The author thanks Monica Musso and Frank Pacard.}

\bibliographystyle{alpha}
\bibliography{bibliography/biblio}

\begin{thebibliography}{{Hof}80}

\bibitem[Ban80]{isoperimetric-inequalities}
Catherine Bandle.
\newblock {\em Isoperimetric Inequalities and Applications}.
\newblock Pitman, London, 1980.

\bibitem[Dup11]{stable-solutions-elliptic}
Louis Dupaigne.
\newblock {\em Stable Solutions of Elliptic Partial Differential Equations}.
\newblock Chapman \& Hall, 2011.

\bibitem[GNN79]{gidas-ni-nirenberg}
Basilis Gidas, Wei-Ming Ni, and Louis Nirenberg.
\newblock Symmetry and related properties via the maximum principle.
\newblock In {\em Commun. Math. Phys. 68, 209-243}. Springer-Verlag, 1979.

\bibitem[{Hof}80]{1980JPhA...13..417H}
T.~{Hoffmann-Ostenhof}.
\newblock {A comparison theorem for differential inequalities with applications
  in quantum mechanics}.
\newblock {\em Journal of Physics A Mathematical General}, 13:417--424,
  February 1980.

\bibitem[Khu04]{bratu}
S.A. Khuri.
\newblock A new approach to bratu's problem.
\newblock In {\em Applied Mathematics and Computation 147(1):131-136}, 2004.

\bibitem[PPS14]{Pacard2014}
Frank Pacard, Filomena Pacella, and Berardino Sciunzi.
\newblock Solutions of semilinear elliptic equations in tubes.
\newblock {\em Journal of Geometric Analysis}, 24(1):445--471, Jan 2014.

\bibitem[PW02]{new-solutions}
M~Plum and C~Wieners.
\newblock New solutions to the gelfand problem.
\newblock In {\em Journal of Mathematical Analysis and Applications 269,
  588-606}, 2002.

\end{thebibliography}


\end{document}